\documentclass[11pt]{article}

\usepackage[latin1]{inputenc}

\usepackage{amsmath,amssymb,amsthm}

\numberwithin{equation}{section}

\theoremstyle{plain}
\newtheorem{theorem}{Theorem}[section]
\newtheorem{proposition}[theorem]{Proposition}
\newtheorem{lemma}[theorem]{Lemma}
\newtheorem{corollary}[theorem]{Corollary}

\newtheorem{fact}[theorem]{Fact}

\theoremstyle{definition}

\newtheorem{remark}[theorem]{Remark}

\newtheorem{definition}[theorem]{Definition}

\DeclareMathOperator{\Ann}{Ann}

\newcommand{\B}{\mathcal{B}}

\def\n3#1{\vert  \! \vert \! \vert \, #1 \, \vert \!
  \vert \! \vert }

\def\lbl#1{\label{#1}}

\def\bb#1{\bibitem{#1}}

\def\R{{\mathbb R}}
\def\N{{\mathbb N}}
\def\C{{\mathbb C}}

\def\K{{\mathbb K}}

\def\B{{\mathbb B}}

\usepackage{bbm}

\begin{document}

\begin{center}{\LARGE \bf Unit-free versions of the Vidav--Palmer \\ theorem and of the \\ Blecher--Ruan--Sinclair non-associative \\ \vspace{0.3cm} characterization of unital $C^*$-algebras}\footnote{AMS Classification 2022: 46L05, 46L70. }
\end{center}

\begin{center}{\Large \'Angel Rodr{\'\i}guez Palacios}\\
\smallskip
\small{Universidad de Granada.\\ Departamento
de An\'{a}lisis Matem\'{a}tico. Facultad de Ciencias.\\
18071 Granada (Spain)
 \\  apalacio@ugr.es}
\bigskip
\end{center}

\begin{center}
To the memory of my wife In\'es Espinar de la Paz
\end{center}

\section*{Abstract}
We prove unit-free versions of both the associative and the non-associative Vidav--Palmer theorems. Then these results are applied to prove a unit-free version of the Blecher--Ruan--Sinclair non-associative characterization of unital $C^*$-algebras.

\section{Introduction}

Let $X$ be a normed space over
$\mathbb K =\R$ or $\C $, and let $u$ be a norm-one element in~$X$. By a {\it state of $X$} relative to $u$ we mean a functional $f$ in the closed unit ball
$\mathbb{B}_{X'}$ of the dual $X'$ of $X$ satisfying $f(u)=1$. We denote by
$D(X,u)$ the set of all states of $X$ relative to $u$, and remark
that, by the Hahn-Banach and Banach-Alaoglu theorems, $D(X,u)$
becomes a non-empty $w^*$-compact convex subset of~$X'$. Given
$x\in X$, we define the  {\it
numerical range} of $x$ relative to $(X,u)$ as the non-empty compact convex
subset of $\mathbb K$ given by
$$V(X,u,x):=\{f(x):f\in D(X,u)\}.$$
In the case that $\K =\C $, an element $x\in X$ is
said to be {\it hermitian} relative to $(X,u)$ if
$V(X,u,x)\subseteq \mathbb R$, and we denote by $H(X,u)$ the closed real subspace of $X$ consisting of all hermitian elements of $X$ relative to $(X,u)$.

To be precise, let us say that, as in \cite{CR,CRbis}, our algebras are not assumed to be associative, and that by a {\it  normed algebra}  we mean a real or complex algebra $A$ endowed with an {\it algebra norm}, i.e. a vector space norm $\Vert \cdot \Vert $ on $A$ such that $\Vert ab\Vert \leq \Vert a\Vert \Vert b\Vert $ for all $a,b\in A$. Nevertheless, according to \cite[$\S$1.2.2]{CR}, our \textit{$C^*$-algebras} without any adjective will always be associative.

Following \cite{CR,P1}, by a \textit{norm-unital normed algebra} we mean a normed algebra having a norm-one unit ${\bf 1}$, and by a \textit{$V$-algebra} we mean a norm-unital normed complex algebra $A$ such that $A$ is equal to the complex linear hull of $H(A,{\bf 1})$. Actually, if $A$ is a $V$-algebra, then the equality
\begin{equation} \label{Vaxion}
A=H(A,{\bf 1})\oplus iH(A,{\bf 1})
\end{equation}
 holds. This follows from the Bohnenblust--Karlin theorem \cite{BK} asserting that the unit of any norm-unital associative normed algebra $A$ is a vertex of $\mathbb{B}_A$, a result which is easily generalized to the general non-associative case \cite[Corollary 2.1.13]{CR}. As a consequence of \eqref{Vaxion}, each $V$-algebra $A$ is endowed with a unique conjugate-linear vector
space involution $*$ (called the \textit{natural involution} of $A$)  such that the set of all $*$-invariant elements of $A$ coincides with $H(A,{\bf 1})$. It is far from being trivial the question of wether or not the natural involution of any $V$-algebra is an algebra involution. Restricting for the moment to the associative case, the answer to this question is affirmative. This is a part of the celebrated Vidav--Palmer theorem, which reads as follows.

\begin{theorem} \label{VP}
{\rm \cite{Vi,Pal1} (cf.  \cite[Lemma 2.2.5 and Theorem 2.3.32]{CR})}
Unital $C^*$-algebras are precisely the complete associative $V$-algebras endowed with their natural involutions.
\end{theorem}

Complete proofs of the above theorem can be found in the books \cite{BD1,BD2,BD,CR,DB,P1,P}.

For the treatment of the Vidav--Palmer theorem in the general non-associative setting, the so-called non-commutative $JB^*$-algebras need to be considered. We recall that
\textit{non-commutative $JB^*$-algebras} are defined as those
complete nor\-med non-commutative Jordan complex $*$-algebras $A$ satisfying
$\Vert U_{a}(a^*)\Vert =\Vert a\Vert ^{3}$ for every $a\in A$,
where $U_{a}(b):=a(ab+ba)-a^{2}b$ for all $a,b \in A$, and that non-commutative $JB^*$-algebras which are commutative
are simply called \textit{$JB^*$-algebras}. According to Wright \cite{W} (who established a categorical correspondence between $JB^*$-algebras and $JB$-algebras \cite{HOS}), $JB^*$-algebras were introduced by Kaplansky in his final lecture to the 1976
St.Andrews Colloquium of the Edinburgh Mathematical Society under the name (obsolete since Youngson's paper \cite{Y0}) of `Jordan $C^*$-algebras'. Actually, without suggesting any definition, Kaplansky's Jordan $C^*$-algebras were previously considered at the end of \cite{70K} (see \cite[p. xv]{CR} and the comments therein). Roughly speaking, non-commutative $JB^*$-algebras can be described by means of the so-called representation theory \cite{AJ,Br1,PPR,PPR1} (cf. \cite[Corollary 6.1.11 and Theorem 6.1.112]{CRbis}). Note finally that, in a very precise sense, non-commutative $JB^*$-algebras become the largest non-associative generalization of $C^*$-algebras \cite{KMR} (cf. \cite[Theorem 5.9.9 and Corollary~5.9.12]{CRbis}).

After successive contributions in \cite{Y0,Y,MMR,R10}, the ``non-associative Vidav--Palmer theorem" reads as follows.

\begin{theorem} \label{non-associativeVP}
{\rm \cite[Lemma 2.2.5 and Theorem 3.3.11]{CR}}
Unital non-com\-mutative $JB^*$-algebras are precisely the complete {\rm (possibly non-associative)} $V$-algebras endowed with their natural involutions.
\end{theorem}

Among other tools, the proof of the above theorem applies that the natural involution of any $V$-algebra is an algebra involution \cite{R10} (cf. \cite[Theorem 2.3.8]{CR}), and that then every $V$-algebra is a non-commutative Jordan algebra \cite{MMR} (cf. \cite[Theorem  2.4.11]{CR}), thus showing how, surprisingly, the non-commutative Jordan identities $(ab)a=a(ba)$ and $(a^2\bullet b)\bullet a=a^2\bullet (b\bullet a)$ follow from natural analytic requirements. We note that Theorem \ref{non-associativeVP} contains Theorem \ref{VP} because $C^*$-algebras are precisely those non-commutative $JB^*$-algebras which are associative \cite[Fact 3.3.2]{CR}.

In the current paper we prove unit-free versions of Theorems \ref{VP} and \ref{non-associativeVP} (see Definitions \ref{definition1} and \ref{definition2}, and Theorems \ref{unit-freeVP} and \ref{unit-freenon-associativeVP}), which contain their unital forerunners (see Lemma \ref{4.100}(i)). These results significantly improve
\cite[Theorem 4.18]{R99} and the paragraph following it.

Now let us deal with the Blecher--Ruan--Sinclair non-associative characterization of unital $C^*$-algebras.

Let $X$ be a vector space over $\K$, and let $n$ be a natural
number. Then the vector space
$M_n(X)$, of all $n\times n$ matrices with entries in $X$, becomes
an $M_n(\K )$-bimodule, by denoting $\alpha x$ and
$x\alpha $ (for $\alpha \in M_n(\K )$ and $x\in M_n(X)$) the
elements of $M_n(X)$ formed by left and right multiplication
of~$x$ by~$\alpha $ in the obvious sense of matrix multiplication.
Given $p,q\in \N$ with $p+q=n$, $y\in M_p(X)$, and $z\in M_q(X)$,
we denote by $y\oplus z$ the element of $M_n(X)$ whose principal
diagonal blocks (from top to the bottom) are $y$ and~$z$ and whose
off-diagonal blocks have~$0$ (in $X$) at each entry.  In what follows, $M_n(\C)$ will be seen endowed with the
involution and the corresponding $C^*$-norm
$\vert \cdot \vert _n$ deriving from its natural
identification with the algebra of all linear operators on the
complex Hilbert space $\C ^n$.

\begin{definition} \label{definition0}
{\rm Let $A$ be a normed complex algebra. We say that $A$ \textit{enjoys a matricial  $L_\infty $-structure}
if, for each $n\in \N$, the complex algebra $M_n(A)$ is equipped with an algebra
norm  $\Vert \cdot \Vert _n$ in such a way that the following conditions are satisfied:
\begin{enumerate}
\item  [\rm (i)] $ \Vert \cdot \Vert _1=\Vert \cdot \Vert $ on $M_1(A)=A$,
\item [\rm (ii)] $\Vert \alpha a\beta \Vert _n\leq \vert \alpha \vert _n\Vert
a \Vert _n\vert \beta \vert _n$ for all $n\in \N$, $a\in M_n(A)$, and
$\alpha ,\beta \in M_n(\C)$,
\item [\rm (iii)] $\Vert b\oplus c\Vert
_{n+m}=\max \{\Vert b\Vert _n,\Vert c\Vert _m\}$ for all $n,m\in \N$, $b\in
M_n(A)$, and $c\in M_m(A)$.
\end{enumerate}
A less demanding situation is that of the existence of an algebra norm $\Vert \cdot \Vert _2$ on
$M_2(A)$ satisfying
\begin{enumerate}
\item [\rm (iv)] $\Vert
\alpha a\beta \Vert _2\leq \vert \alpha \vert _2\Vert a \Vert
_2\vert \beta \vert _2$ for every $a\in M_2(A)$ and all $\alpha
,\beta \in M_2(\C)$,
\item [\rm (v)] $\Vert b\oplus c\Vert _2=\max \{\Vert
b\Vert ,\Vert c\Vert \}$ for all $b,c\in A$.
\end{enumerate}
In this case we say that $A$ \textit{enjoys a matricial  $L_\infty ^2$-structure}. Let us finally recall that, according to \cite[$\S$2.4.26]{CR}, {\it non-associative
$C^*$-algebras} are defined by altogether removing in the definition of
$C^*$-algebras any requirement of associativity.}
\end{definition}

Now, with some refinements obtained in \cite{CR}, the Blecher--Ruan--Sinclair non-associative characterization of unital $C^*$-algebras reads as follows.

\begin{theorem} \label{BlRuSi}
{\rm \cite[Corollary 3.3]{BlRuSi} (cf.  \cite[Theorem  2.4.27]{CR})}
Let $A$ be a normed complex
algebra. Then the following conditions are equivalent:
\begin{enumerate}
\item [\rm (i)] $A$ is a unital non-associative  $C^*$-algebra enjoying a matricial  $L_\infty $-structure.
\item  [\rm (ii)] $A$ is a complete $V$-algebra endowed with its natural involution, and enjoys a matricial $L_\infty
$-structure.
\item  [\rm (iii)] $A$ is a unital non-associative
$C^*$-algebra  enjoying  a matricial  $L_\infty ^2$-structure.
 \item [\rm (iv)]
$A$ is a complete $V$-algebra endowed with its natural involution, and enjoys a matricial $L_\infty
^2$-structure.
\item [\rm (v)]  $A$ is a unital $C^*$-algebra.
\end{enumerate}
\end{theorem}

In relation to the above theorem, it is worth mentioning that, according to \cite[Theorem 2.2]{BlRuSi} (see also the last sentence in \cite[Theorem  2.4.27]{CR}), a unital $C^*$-algebra enjoys a unique matricial  $L_\infty $-structure, just the one given by \cite[Proposition 2.4.22]{CR}.

In the current paper we apply Theorems \ref{unit-freeVP} and \ref{unit-freenon-associativeVP} (already quoted) to prove a unit-free version of Theorem \ref{BlRuSi} (see Theorem~\ref{unit-freeBlRuSi}). We note that our proof of Theorem \ref{unit-freeBlRuSi} is complete, meaning this that its unital forerunner (namely Theorem \ref{BlRuSi}) is never applied. We also prove that any (possibly non-unital) $C^*$-algebra enjoys a unique matricial  $L_\infty $--structure, just the one given by \cite[Proposition~2.4.22]{CR} (see Theorem \ref{final}).

{\it Alternative $C^*$-algebras} are defined as those complete normed
alternative complex $*$-algebras $A$ satisfying the \textit{Gelfand--Naimark axiom}   $\Vert a^*a\Vert =\Vert a\Vert
^2$ for every $a\in A$. As happened with $JB^*$-algebras, alternative $C^*$-algebras were introduced by Kaplansky at the end of \cite{70K} (see \cite[p. xiii--xiv]{CR} and the comments therein).
As a matter of fact, alternative $C^*$-algebras become  particular examples  of non-commu\-tative
$JB^*$-algebras. Indeed, according to  \cite[Fact~3.3.2]{CR}, \textit{alternative $C^*$-algebras are precisely those
non-commutative $JB^*$-algebras which are alternative}. Therefore, as a straightforward consequence of Theorem \ref{unit-freenon-associativeVP}, already quoted, we obtain a unit-free Vidav--Palmer type theorem for alternative $C^*$-algebras (see Corollary \ref{alternativeunit-freeVP}).

Although many results for alternative $C^*$-algebras appear in the literature (including \cite{CR,CRbis})  camouflaged as results for non-commutative $JB^*$-algebras, alternative $C^*$-algebras have their own life \cite{MMR,PPR,R11,R99,R22,Br,CRtris,CaR11,CR16,Ho,R5,R88}, mainly due to the fact that, in a very precise sense, they become the largest non-associative generalizations of $C^*$-algebras when these are defined by means of the Gelfand--Naimark axiom \cite{R11} (cf.~\cite[Theorem 3.5.53]{CR}).

Concerning the unit-free version of the Vidav--Palmer theorem, the peculiar behaviour of alternative $C^*$-algebras, as particular examples of non-commutative $JB^*$-algebras, is explored in Section 4. Thus in Theorem \ref{finalfinal} we prove a very strong refinement of Corollary \ref{alternativeunit-freeVP} quoted above. We note that Theorem \ref{finalfinal} also refines \cite[Theorem 4.20]{R99}. We also note that the proof of Theorem \ref{finalfinal} does not depend on the results proved in Sections 2 and 3. It goes without saying that Theorem \ref{finalfinal} has interesting consequences in the associative setting (see Corollaries \ref{unit-freeVPbis} and~\ref{unit-freeVPtris}).

Left operator $V$-algebras introduced in Definition \ref{definition1} play a central role throughout this paper. Thus in Section 5 we begin to develop a general theory of such algebras. Indeed, operator $V$-algebras are stable by passing to quotients by closed ideals, and by passing to $\ell _\infty $-sums (see Propositions \ref{Videal} and \ref{ya}). Moreover  operator $V$-algebras behave reasonably well concerning the so-called contractive projection theorem (see Theorem~\ref{martes}). Some applications of our results to non-commutative $JB^*$-algebras, like Theorem \ref{new}, could be new even in the case of $C^*$-algebras. Moreover, a ``unit-free" and ``completeness-free" Vidav--Palmer type theorem for alternative algebras (so also for associative algebras) is obtained (see Theorem \ref{completeness-freeVP} and Corollary~\ref{unit-freeVP00}).

\section{The unit-free version of the Vidav--Palmer theorem}

Given an element $a$ in an algebra $A$, we denote by $L_a^A$ (respectively $R_a^A$) the operator of left (respectively, right) multiplication by $a$ on $A$. Now let $A$ be a normed algebra over $\K$.
By a \textit{left approximate identity} in $A$ we mean a net $u_\lambda $ in $A$ such that $\lim _\lambda u_\lambda a=a$ for every $a\in A$. \textit{Right approximate identities} in $A$ are defined analogously, and (two-sided) \textit{approximate identities} in $A$ are defined as those nets in $A$ which are both left and right approximate identities in $A$. If $A\neq 0$, then we define the \textit{left operator numerical range} of an element $a\in A$ relative to $A$  by $LOV(A,a):=V(BL(A),I_A,L_a^A)$.

The key tool in this section is the following lemma, most part of which was proved in \cite{R11} under the unnecessary requirements that the normed algebra is complete and complex.

\begin{lemma} \label{4.2}
Let $A$ be a nonzero normed algebra over $\K$, and let $a_\lambda $
be an approximate identity bounded by $1$ in $A$. Consider $A''$ as a
complete normed algebra under the Arens product, let~${\bf 1}$
be a $w^*$-cluster point in~$A''$ to the net~$a_\lambda$, and set $B:=\mathbb K{\bf 1}+A\subseteq A''$. Then we
have:
\begin{enumerate}
\item[\rm (i)]
$B$ becomes a norm-unital
subalgebra of~$A''$ with unit ${\bf 1}$.
\item[\rm (ii)]
For $\alpha \in \mathbb K$ and $a\in A$ the equality $\Vert \alpha
{\bf 1}+a\Vert =\lim _{\lambda} \Vert \alpha a_\lambda +a\Vert $ holds.
\item[\rm (iii)]
For $a\in A$ the equality $LOV(A,a)=V(B,{\bf 1},a)$ holds.
\end{enumerate}
\end{lemma}
\begin{proof}
Considering \cite{R11} (cf. \cite[Lemma 3.5.46]{CR}), it only remains to prove assertion (iii). To this end, we keep in mind assertion (i), and argue as follows.

Let $a$ be in $A$. Then, by  \cite[Lemmas 2.2.24 and 2.1.10]{CR}, we have
\[
LOV(A,a)\subseteq LOV(B,a)=V(B,{\bf 1},a).
\]
 On the other hand, since ${\bf 1}$ is a $w^*$-cluster point in~$A''$ to the net~$a_\lambda$, and $(L_a^A)'':A''\to A''$ extends $L_a^A$ and is $w^*$-continuous, and $a_\lambda $
is an approximate identity in $A$, we realize that $(L_a^A)''({\bf 1})=a$. Therefore $B$ is invariant under $(L_a^A)''$, and $((L_a^A)'')_{|B}$ (regarded as a mapping from $B$ to $B$) is equal to $L_a^B$. It follows from  \cite[Lemmas 2.1.10 and 2.2.24, and Corollary 2.1.3]{CR} that
\[
LOV(B,a)=V(BL(B),I_B,((L_a^A)'')_{|B})
\]
\[
\subseteq V(BL(A''),I_{A''},(L_a^A)'')=LOV(A,a).
\]
 Thus, by means of the double inclusion, the equality $LOV(A,a)=V(B,{\bf 1},a)$
  has been proved. Since $a$ is arbitrary in $A$, the proof of assertion (iii) and of the lemma is concluded.
 \end{proof}

\begin{definition} \label{definition1}
{\rm Let $A$ be a \textit{nonzero} normed complex algebra. We say that an element $a\in A$ is \textit{left operator hermitian} if $L_a^A$ belongs to $H(BL(A),I_A)$, and denote by $LOH(A)$ the closed real subspace of $A$ consisting of all left operator hermitian elements of $A$. If $A$ is equal to the complex linear hull of $LOH(A)$, then we say that $A$ is a \textit{left operator $V$-algebra.}}
\end{definition}

The notion of a left operator $V$-algebra, just introduced, is a weakening of other notions previously appeared in the literature (see \cite[Definition~3.8]{R99} and references therein).

\begin{lemma} \label{4.100}
The following assertions hold:
\begin{enumerate}
\item[\rm (i)]
For a norm-unital complex algebra $A$, the equality $LOH(A)=H(A,{\bf 1})$ holds, and hence $V$-algebras are precisely the norm-unital left operator $V$-algebras.
\item[\rm (ii)]
If $A$ is a left operator $V$-algebra having a right approximate identity, then the equality $A=LOH(A)\oplus iLOH(A)$ holds,
and hence $A$ is endowed with a unique conjugate-linear vector
space involution $*$  such that the set of all $*$-invariant elements of $A$ coincides with $LOH(A)$.
\end{enumerate}
\end{lemma}
\begin{proof}
Assertion (i) follows from  \cite[Lemma 2.1.10]{CR}.

Now let $A$ be a nonzero normed complex algebra having a right approximate identity (say $a_\lambda $), and let $a$ be in $LOH(A)\cap iLOH(A)$. Then $L_a^A$ belongs to $H(BL(A),I_A)\cap iH(BL(A),I_A)$, and hence, by \cite[Corollary 2.1.13]{CR}, $L_a^A=0$. Therefore $a=\lim _\lambda aa_\lambda =\lim _\lambda L_a^A(a_\lambda )=0$. Thus $LOH(A)\cap iLOH(A)=0$, and (ii) follows.
 \end{proof}

\begin{definition} \label{definition2}
{\rm Let $A$ be a left operator $V$-algebra having a right approximate identity. Then the involution $*$ given by assertion (ii) in Lemma \ref{4.100} will be called \textit{the natural involution} of $A$.}
\end{definition}

Now the unit-free version of Theorem \ref{VP} reads as follows.

\begin{theorem} \label{unit-freeVP}
Nonzero $C^*$-algebras are precisely the complete associative left operator $V$-algebras having an approximate identity bounded by $1$, endowed with their natural involutions.
\end{theorem}
\begin{proof}
Nonzero $C^*$-algebras have approximate identities bounded by $1$ \linebreak \cite[Lemma 39.14]{BD}. Moreover, by \cite[Lemma 3.6.24]{CR}, they are complete associative left operator $V$-algebras  whose natural involutions coincide with their involutions as $C^*$-algebras. This is \textit{the easy part} of the theorem.

Suppose that $A$ is a complete associative left operator $V$-algebra and has an approximate identity $a_\lambda $ bounded by $1$. Let ${\bf 1}$ and $B$ be as in Lemma~\ref{4.2}. Suppose that ${\bf 1}\in A$. Then, by Lemma \ref{4.2}(i), $A=B$ is a norm-unital complete associative left operator $V$-algebra. Therefore, by Lemma \ref{4.100}(i) and Theorem \ref{VP}, $A$ is a $C^*$-algebra. Finally suppose that ${\bf 1}\notin A$. Then \linebreak $B=\C {\bf 1}\oplus A$, and it follows from Lemma \ref{4.2}(iii) that $B$ is a complete associative $V$-algebra whose natural involution (as a $V$-algebra) is the mapping $\alpha {\bf 1}+a\to \overline{\alpha}{\bf 1}+a^*$, where $*$ denotes the natural involution of $A$ (as a left operator $V$-algebra having a right approximate identity). Therefore, by Theorem \ref{VP}, $B$ is a $C^*$-algebra for the involution pointed out above. Since $A$ is invariant under such an involution, $A$ is a $C^*$-algebra under its natural involution.
\end{proof}

Noticing that non-commutative $JB^*$-algebras have approximate identities bounded by $1$ \cite{Vt} (cf. \cite[Proposition 3.5.23]{CR}), we can argue as in the above proof, with Theorem \ref{non-associativeVP} instead of Theorem \ref{VP}, to obtain the following.

\begin{theorem} \label{unit-freenon-associativeVP}
Nonzero non-commutative $JB^*$-algebras are precisely the complete {\rm (possibly non-associative)} left operator $V$-algebras  having an approximate identity bounded by $1$, endowed with their natural involutions.
\end{theorem}

Note that, thanks to Lemma \ref{4.100}(i), Theorem \ref{unit-freeVP} (respectively, Theorem \ref{unit-freenon-associativeVP}) contains its unital forerunner, namely Theorem \ref{VP} (respectively, Theorem \ref{non-associativeVP}).

As we already know, alternative $C^*$-algebras are precisely those
non-commutative $JB^*$-algebras which are alternative. Therefore, as consequence of Theorem \ref{unit-freenon-associativeVP} we obtain the following generalization of Theorem \ref{unit-freeVP}.

\begin{corollary} \label{alternativeunit-freeVP}
Nonzero alternative $C^*$-algebras are precisely the complete alternative left operator $V$-algebras having an approximate identity bounded by $1$, endowed with their natural involutions.
\end{corollary}

\begin{remark} \label{nota}
{\rm Let $A$ be a nonzero normed complex algebra and, parodying Definition \ref{definition1}, let us say that an element $a\in A$ is \textit{right operator hermitian} if $R_a^A$ belongs to $H(BL(A),I_A)$, let us denote by $ROH(A)$ the closed real subspace of $A$ consisting of all right operator hermitian elements of $A$, and let us say that $A$ is a \textit{right operator $V$-algebra} if $A$ is equal to the complex linear hull of $ROH(A)$. Then, arguing as above, we can prove the appropriate variants of Lemmas \ref{4.2}(iii) and \ref{4.100}.
Analogously, parodying Definition~\ref{definition2}, we can define the \textit{natural involution} of a right operator $V$-algebra having a left approximate identity, and can prove the appropriate variants of Theorems \ref{unit-freeVP} and \ref{unit-freenon-associativeVP}, and of Corollary \ref{alternativeunit-freeVP}. As a result, we obtain that \textit{nonzero $C^*$-algebras $($respectively, non-commutative $JB^*$-algebras, or alternative $C^*$-algebras$)$ are precisely the complete associative $($respectively, possibly non-associative, or alternative$)$ right operator $V$-algebras having an approximate identity bounded by $1$, endowed with their natural involutions.} As a byproduct, combining this with Theorem \ref{unit-freenon-associativeVP}, we obtain that \textit{complete left operator $V$-algebras having an approximate identity bounded by $1$ coincide with complete right operator $V$-algebras having an approximate identity bounded by $1$, and the natural involution of such an algebra as a left operator algebra having a right approximate identity coincides with the natural involution  as a right operator algebra having a left approximate identity.} Thus the ``asymmetry" of our results is only apparent.}
\end{remark}

Despite the above remark, an explicit ``symmetric" version of the unit-free Vidav--Palmer theorem can be proved. To establish it, we recall that, for any  real or complex algebra $A$, $A^{\rm sym}$ denotes the algebra consisting of the vector space of $A$ and the product $a\bullet b:=\frac{1}{2}(ab+ba)$, and that, if $A$ is a normed algebra, then so is $A^{\rm sym}$ under the norm of $A$. With these conventions, the announced symmetric version of the unit-free Vidav--Palmer theorem reads as follows.

\begin{theorem}
For a complete normed complex algebra $A$, the following conditions are equivalent:
\begin{enumerate}
\item[\rm (i)] $A$ is a nonzero non-commutative $JB^*$-algebra.
\item[\rm (ii)] $A$ has an approximate identity bounded by $1$, and $A^{\rm sym}$ is a left \linebreak $($$=$ right by commutativity$)$ operator $V$-algebra.
\item[\rm (iii)] $A^{\rm sym}$ is a left $($$=$ right$)$ operator $V$-algebra and has an approximate identity  bounded by $1$.
\end{enumerate}
\end{theorem}
\begin{proof}
(i)$\Rightarrow $(ii) Suppose that (i) holds. Then, by \cite[Proposition 3.5.23]{CR}, $A$ has an approximate identity bounded by $1$. On the other hand, by \cite[Fact~3.3.4]{CR}, $A^{\rm sym}$ is a $JB^*$-algebra, and hence, by the easy part of Theorem \ref{unit-freenon-associativeVP}, $A^{\rm sym}$ is a left operator $V$-algebra.

(ii)$\Rightarrow $(iii) Clearly, every approximate identity bounded by $1$ in $A$ is an approximate identity bounded by $1$ in $A^{\rm sym}$.

(iii)$\Rightarrow $(i) Suppose that (iii) holds. Then, by Theorem \ref{unit-freenon-associativeVP}, $A^{\rm sym}$ is a $JB^*$-algebra. Therefore, by \cite{R91} (cf. \cite[Definition 3.5.29 and Proposition 3.5.31]{CR}), $A$ is a non-commutative $JB^*$-algebra.
\end{proof}

Since $C^*$-algebras (respectively, alternative $C^*$-algebras)  are precisely  those non-commutative $JB^*$-algebras which are associative (respectively, alternative), the above theorem provides us with a symmetric unit-free version of the Vidav--Palmer theorem for $C^*$-algebras (respectively, alternative $C^*$-algebras).

\section{The unit-free version of the Blecher--Ruan--Sin\-clair non-asso\-ciative characterization of unital $C^*$-al\-gebras}

The key tool in this section is the following.

\begin{lemma} \label{4.2bis}
For $i=1,2$, let $A_i$ be a nonzero nor\-med algebra over $\K $, and let $a_{i,\lambda }$
be an approximate identity bounded by $1$ in $A_i$. Let $\phi :A_1\to A_2$ be a contractive linear mapping such that $\phi (a_{1,\lambda })=a_{2,\lambda }$ for every $\lambda $. Then for every $x\in A_1$ we have
$LOV(A_2,\phi (x))\subseteq LOV(A_1,x)$.
\end{lemma}
\begin{proof}
For $i=1,2$ let ${\bf 1}_i$ and $B_i$ the element of $A_i''$ and the subalgebra of $A_i''$, respectively, given by Lemma \ref{4.2}. Then, by assertion (ii) in that lemma, for $(\alpha ,x)\in \K \times A_1$ we have
\[
\mbox{$\|\alpha {\bf 1}_2+\phi (x)\|=\lim _{\lambda} \|\alpha a_{2,\lambda}+\phi (x)\|=\lim _{\lambda} \|\phi (\alpha a_{1,\lambda}+x)\|$}
\]
\[
\mbox{$\leq \lim _{\lambda} \|\alpha a_{1,\lambda}+x\|=\|\alpha {\bf 1}_1+x\|$.}
\]
The inequality $\|\alpha {\bf 1}_2+\phi (x)\|\leq \|\alpha {\bf 1}_1+x\|$ just proved for every $(\alpha ,x)\in \K \times A_1$ implies that $\alpha {\bf 1}_1+x\to \alpha {\bf 1}_2+\phi (x)$ becomes a contractive linear well-defined mapping from $B_1$ to $B_2$ taking ${\bf 1}_1$ to ${\bf 1}_2$.  Therefore, by \cite[Corollary 2.1.2]{CR}, we have $V(B_2,{\bf 1}_2,\phi (x))\subseteq V(B_1,{\bf 1}_1,x)$ for every $x\in A_1$, and the proof is concluded by appling Lemma \ref{4.2}(iii).
\end{proof}

\begin{proposition} \label{propisition1}
Let $A$ be a complete normed complex algebra, let $a_\lambda $ be an approximate identity bounded by $1$ in $A$, and let $T$ be in $H(BL(A),I_A)$ such that
$T(a_\lambda )=0$ for every $\lambda $. Then $iT(LOH(A)) \subseteq LOH(A)$.
\end{proposition}
\begin{proof}
Let $r$ be a nonzero real number. Then, since $T \in H(BL(A), I_A)$, it follows from \cite[Corollary 10.13]{BD} (see also \cite[Corollary 2.1.9(iii)]{CR})
that  $\exp (irT) $ is a linear isometry. On the
other hand, since $T(a_\lambda )=0$ for every $\lambda $, we have $\exp (irT) (a_\lambda )=a_\lambda $ for every $\lambda $. Therefore, by Lemma \ref{4.2bis}, for every $a\in A$ we have
$LOV(A,\exp (irT)(a))\subseteq LOV(A,a)$.
As a consequence, we derive that
$\exp (irT) (LOH(A)) \subseteq LOH(A)$.
Keeping in mind that
$LOH(A)$ is a closed real subspace of $A$, and that
$$iT(a)= \lim_{r \to 0} \frac{1}{r} [( \exp (irT))(a) -a]$$ for
every $a \in A$, we obtain $iT(LOH(A)) \subseteq LOH(A)$, as
required.
\end{proof}

Now, let $X$ be a complex vector space endowed with a conjugate-linear involution $*$, and let~$n$ be in $\N$.
Then $M_n(X)$ has a {\it canonical} conjugate-linear {\it
involution}, namely the one consisting of
transposing the matrix and applying $*$ to each
entry.

\begin{lemma} \label{estupendo}
Let $A$ be a nonzero complete left operator $V$-algebra having an approximate identity bounded by $1$ and enjoying a matricial $L_\infty
^2$-structure. Then $(M_2(A),\|\cdot \|_2)$ is a complete left operator $V$-algebra having an approximate identity bounded by $1$, and the natural involution of $M_2(A)$ is the canonical involution corresponding to the natural involution of $A$.
\end{lemma}
\begin{proof}
Let $*$ denote the natural involution of $A$ (cf. Definition \ref{definition2}), and let $a_\lambda $ be the approximate identity in $A$ bounded by $1$ whose existence has been assumed. Then, by Theorem \ref{unit-freenon-associativeVP} and \cite{Y0} (cf. \cite[Proposition~3.3.13]{CR}), $*$ is an isometric algebra involution on $A$, and hence $b_\lambda :=\frac{1}{2}(a_\lambda +a_\lambda ^*)$ becomes an approximate identity bounded by $1$ in $A$ such that $b_\lambda \in LOH(A)$ for every $\lambda $. Therefore we may and will assume that $a_\lambda \in LOH(A)$ for every~$\lambda $.
Note also that, by \cite[Lemma 2.4.21]{CR} and condition (v) in Definition \ref{definition0}, $(M_2(A),\|\cdot \|_2)$ is a complete normed algebra, and $a_\lambda \oplus a_\lambda $ is an approximate identity bounded by $1$ in $M_2(A)$. Now we are going to prove the following.
\begin{enumerate}
\item[\bf \\ Claim] \textit{ $$\mbox{$\begin{pmatrix} h & 0 \\
0 & 0 \end{pmatrix}$, $\begin{pmatrix} 0 & 0 \\
0  & h \end{pmatrix}$, $\begin{pmatrix} 0 & h \\
h  & 0 \end{pmatrix}$, \ and \ $\begin{pmatrix} 0 & ih \\
-ih  & 0 \end{pmatrix}$} $$  lie in \ $LOH(M_2(A))$ whenever $h$ belongs to $LOH(A)$.}
\end{enumerate}
 To prove the claim we begin by applying Lemma \ref{4.2} to both $(A,a_\lambda )$ and $(M_2(A),a_\lambda \oplus a_\lambda )$. We denote by ${\bf 1}_2$ and $B_2$ the element and the subalgebra of $M_2(A)''$ given by Lemma \ref{4.2} in the case that $(M_2(A),a_\lambda \oplus a_\lambda )$ replaces $(A,a_\lambda )$. Let $h$ be in $LOH(A)$. Then, by Lemma \ref{4.2}(ii), for every $r>0$ we have $\Vert
{\bf 1}+irh\Vert =\lim _{\lambda} \Vert a_\lambda +irh\Vert $, and also, considering condition (v) in Definition \ref{definition0}, we have
\[
\Vert
\mbox{${\bf 1}_2+ir(h\oplus 0)\Vert _2=\lim _{\lambda} \Vert a_\lambda \oplus a_\lambda +ir(h\oplus 0)\Vert _2$}
\]
\[
\mbox{$=\lim _{\lambda} \Vert (a_\lambda +irh)\oplus a_\lambda \Vert _2=\lim _{\lambda} \max \{\|a_\lambda +irh\|,\|a_\lambda \|\}$.}
\]
But $\|a_\lambda \|\leq \|a_\lambda +irh\|$ for every $\lambda $ because $a_\lambda $ and $rh$ belongs to $LOH(A)$, and $*$ is an isometric conjugate-linear involution on $A$. It follows that
\[
\Vert
{\bf 1}+irh\Vert =\Vert
{\bf 1}_2+ir(h\oplus 0)\Vert _2,
\]
 and, replacing $h$ with $-h$, also
\[
\Vert
{\bf 1}-irh\Vert =\Vert
{\bf 1}_2-ir(h\oplus 0)\Vert _2
\]
 Therefore, since $r$ is an arbitrary positive real number, and $h$ belongs to $LOH(A)$ ($=H(B,{\bf 1})$ by Lemma \ref{4.2}(iii)), we derive from \cite{Lu} (cf. \cite[Proposition 2.1.5]{CR}) that
$\begin{pmatrix} h & 0 \\
0 & 0 \end{pmatrix}=h\oplus 0\in H(B_2,{\bf 1}_2)=LOH(M_2(A))$.
Analogously, $\begin{pmatrix} 0 & 0 \\
0  & h \end{pmatrix}\in LOH(M_2(A))$.
Let $\alpha $ be in $M_2(\C )$, and let $l _\alpha $ and $r_\alpha
$ stand for the operators of left and right multiplication by
$\alpha $ on the $M_2(\C )$-bimodule $M_2(A)$. By condition~(iv) in Definition \ref{definition0}, the
mappings $\alpha \to l_\alpha $  and $\alpha \to r_\alpha $ from
$M_2(\C )$ to $BL(M_2(A))$ are linear contractions taking the unit
${\bf 1}$ of $M_2(\C )$ to $I_{M_2(A)}$, and hence, by
\cite[Corollary 2.1.2(i)]{CR}, we have
$$\mbox{$V(BL(M_2(A)),I_{M_2(A)},l_\alpha)\subseteq V(M_2(\C ),{\bf
1},\alpha )$ \ \ and }$$  $$V(BL(M_2(A)),I_{M_2(A)},r_\alpha)\subseteq
V(M_2(\C ),{\bf 1},\alpha ).$$ Therefore $l_\alpha -r_\alpha $
lies in $H(BL(M_2( \! A \! )),I_{M_2( \! A \!)})$ whenever $\alpha $ is in
$H(M_2( \! \C \! ),{\bf 1})$. Note also that  $(l_\alpha
-r_\alpha )(a_\lambda \oplus a_\lambda )=0$ for every $\lambda $. It follows from
Proposition \ref{propisition1} that $i(\alpha x-x\alpha )$ lies in
$LOH(M_2(A))$ whenever $\alpha $ is in $H(M_2(\C ),{\bf
1})$ and $x$ is in
$LOH(M_2(A))$. By taking $$\mbox{$\alpha :=\begin{pmatrix} 0 & i \\
-i  & 0 \end{pmatrix}\in H(M_2(\C ),{\bf
1})$ \ and \ $x:=\begin{pmatrix} h & 0 \\
0  & 0 \end{pmatrix}\in LOH(M_2(A))$}$$
(successively, $\alpha :=\begin{pmatrix} 0 & 1 \\
1  & 0 \end{pmatrix}$  and $x:=\begin{pmatrix} 0 & 0 \\
0  & h \end{pmatrix}$), the above gives that both $\begin{pmatrix} 0 & h \\
h  & 0 \end{pmatrix}$ and $\begin{pmatrix} 0 & ih \\ -ih  & 0
\end{pmatrix}$ lie in $LOH(M_2(A))$, thus concluding the proof of the claim.

Now endow $M_2(A)$ with the canonical involution corresponding to $*$. It follows from the claim just proved that $LOH(M_2(A))$ contains the set of all matrices invariant under such an involution. Therefore $M_2(A)$ is a left operator $V$-algebra, and the natural involution of $M_2(A)$ is the canonical involution corresponding to $*$.
\end{proof}

Now the unit-free version of the Blecher--Ruan--Sinclair non-associative characterization of unital $C^*$-algebras reads as follows.

\begin{theorem} \label{unit-freeBlRuSi}
Let $A$ be a nonzero normed complex
algebra. Then the following conditions are equivalent:
\begin{enumerate}
\item [\rm (i)] $A$ is a non-associative  $C^*$-algebra having an approximate identity boun\-ded by $1$ and enjoying a matricial  $L_\infty $-structure.
\item [\rm (ii)] $A$ is a complete left operator $V$-algebra having an approximate identity bounded by $1$ $($endowing $A$ with its natural involution$)$, and enjoying a matricial $L_\infty
$-structure.
\item [\rm (iii)] $A$ is a non-associative
$C^*$-algebra having an approximate identity boun\-ded by $1$ and enjoying a matricial  $L_\infty ^2$-structure.
 \item [\rm (iv)]
$A$ is a complete left operator $V$-algebra having an approximate identity bounded by $1$ $($endowing $A$ with its natural involution$)$, and enjoying a matricial $L_\infty
^2$-structure.
\item [\rm (v)] $A$ is a $C^*$-algebra.
\end{enumerate}
\end{theorem}
\begin{proof}
The implications (i)$\Rightarrow $(iii) and (ii)$\Rightarrow $(iv)
are clear.

(iii)$\Rightarrow $(iv) By \cite{R11} (cf. \cite[Theorem 3.5.53]{CR}), $A$ is an alternative $C^*$-algebra, and then, by the easy part of Corollary \ref{alternativeunit-freeVP}, $A$ is a left operator $V$-algebra.

(v)$\Rightarrow $(i) By \cite[Propositions 3.5.23 and 2.4.22]{CR}.

(v)$\Rightarrow $(ii) By the easy part of Theorem \ref{unit-freeVP} and \cite[Proposition 2.4.22]{CR}.

Now to conclude the proof it only remains to prove that (iv) implies~(v).

Suppose that (iv) holds. Then, by Lemma \ref{estupendo}, $(M_2(A),\|\cdot \|_2)$ is a complete left operator $V$-algebra having an approximate identity bounded by $1$. Therefore,
by Theorem \ref{unit-freenon-associativeVP}, $M_2(A)$ is a non-commutative Jordan algebra. Thus $M_2(A)$ is a flexible algebra, which implies, by \cite[Lemma 2.4.23]{CR}, that $A$ is associative. Then, condition~(v) follows from Theorem \ref{unit-freeVP}.
\end{proof}

We note that, thanks to Lemma \ref{4.100}(i), Theorem \ref{unit-freeBlRuSi} contains its unital forerunner, namely Theorem \ref{BlRuSi}. Note also that, considering Remark \ref{nota}, we may realize that Theorem \ref{unit-freeBlRuSi} remains true if we replace ``left" with ``right" in conditions (ii) and (iv).

Combining Theorems \ref{unit-freenon-associativeVP} and \ref{unit-freeBlRuSi}, we obtain the following.

\begin{corollary} \label{sinetiqueta}
Let $A$ be a nonzero normed complex
algebra. Then the following conditions are equivalent:
\begin{enumerate}
\item [\rm (i)] $A$ is a non-commutative
$JB^*$-algebra enjoying a matricial  $L_\infty $-structure.
\item [\rm (ii)] $A$ is a non-commutative
$JB^*$-algebra enjoying a matricial  $L_\infty ^2$-structure.
\item [\rm (iii)] $A$ is a $C^*$-algebra.
\end{enumerate}
\end{corollary}

Given an algebra $A$ over $\K$, by a \textit{Jordan subalgebra} of $A$ we mean a subalgebra of $A^{\rm sym}$. Norm-closed $*$-subalgebras of the $JB^*$-algebra $A^{\rm sym}$, for some $C^*$-algebra $A$, are called \textit{$JC^*$-algebras}.

Considering  Corollary \ref{sinetiqueta} and the commutative Gelfand-Naimark theorem, we obtain the following.

\begin{corollary}
Let $A$ be a nonzero normed complex
algebra. Then the following conditions are equivalent:
\begin{enumerate}
\item [\rm (i)] $A$ is a
$JB^*$-algebra enjoying a matricial  $L_\infty $-structure.
\item [\rm (ii)] $A$ is
$JB^*$-algebra enjoying a matricial  $L_\infty ^2$-structure.
\item [\rm (iii)] $A$ is a
$JC^*$-algebra enjoying a matricial  $L_\infty $-structure.
\item [\rm (iv)] $A$ is
$JC^*$-algebra enjoying a matricial  $L_\infty ^2$-structure.
\item [\rm (v)] $A$ is the $C^*$-algebra of all complex-valued continuous functions on some locally compact Hausdorff topological space, vanishing at infinity.
\end{enumerate}
\end{corollary}

\begin{remark}
(a) Let $A$ be a complete normed algebra over $\K$ having an approximate identity bounded by $1$, and enjoying a matricial  $L_\infty $-structure. It follows from Ruan's theorem \cite[Theorem 3.1]{Rua} (cf. \cite[Theorem 1.2.13]{BlLeM}) and \cite[Theorem 2.3.2]{BlLeM} that $A$ is a closed subalgebra of $BL(H)$, for some Hilbert space $H$ over $\K$, and that the given  matricial $L_\infty $-structure of $A$ is nothing other than the one inherited from the natural matricial $L_\infty $-structure of $BL(H)$. Summarized in the classical terminology: $A$ is an \textit{operator algebra} (see \cite[p. 49]{BlLeM}). Considering this, it is not too hard to realize that the implication (ii)$\Rightarrow$(v) in Theorem~\ref{unit-freeBlRuSi} is true.

(b) It follows from Ruan's theorem quoted above and \cite[Proposition 1.2.4]{BlLeM} that a $C^*$-algebra enjoys a unique matricial $L_\infty $-structure.
\end{remark}

\section{Coming back to the unit-free version of the Vi\-dav--Palmer theorem for alternative $C^*$-algebras}

Let $A$ be a real or complex algebra. The \textit{left annihilator} of $A$ is defined as the set of those elements $a\in A$ such that $aA=0$, and is denoted by $l$-$\Ann (A)$. We note that  $l$-$\Ann (A)$ is nothing other than the set of those elements $a\in A$ such tat $L_a^A=0$.

We begin this section with a few elementary observations, which are picked up in the following.

\begin{fact} \label{fact00}
For a normed algebra $A$ over $\K$, the following assertions hold:
\begin{enumerate}
\item[\rm (i)] If there is a right approximate identity in $A$, then $l$-$\Ann (A)=0$.
\item[\rm (ii)] If there is a right approximate identity bounded by $M\geq 1$ in $A$, then for every $a\in A$ we have $\|L_a^A\|\leq \|a\|\leq M\|L_a^A\|$.
\item[\rm (iii)] If $\K =\C$, and if $A\neq 0$, then $LOH(A)\cap iLOH(A)=l$-$\Ann (A)$.
\end{enumerate}
\end{fact}
\begin{proof}
(i) Suppose that $a_\lambda $ is a right approximate identity in $A$. Then, for \linebreak $a\in l$-$\Ann (A)$, we have $a=\lim _\lambda aa_\lambda =\lim _\lambda L_a^A(a_\lambda )=0$.

(ii) Suppose that $a_\lambda $ is a right approximate identity bounded by $M\geq 1$ in $A$, and let $a$ be in $A$. Then we have $\|a\|=\lim _\lambda \|aa_\lambda \|=\lim _\lambda \|L_a^A(a_\lambda )\|$. But $ \|L_a^A(a_\lambda )\|\leq M\|L_a^A\|$ for every $\lambda $. Therefore $\|a\|\leq M\|L_a^A\|$. Finally, the inequality $\|L_a^A\|\leq \|a\|$ is clear.

(iii) Suppose that $\K =\C$ and that $A\neq 0$. Then the inclusion

\[
\mbox{$l$-$\Ann (A)\subseteq LOH(A)\cap iLOH(A)$}
\]
 is clear. Let $a$ be in $LOH(A)\cap iLOH(A)$. Then
 \[
 L_a^A\in H(BL(A),I_A)\cap iH(BL(A),I_A),
 \]
  and hence, by \cite[Corollary 2.1.13]{CR}, $L_a^A=0$, i.e. $a$ lies in  $l$-$\Ann (A)$.
\end{proof}

Let $A$ be a left operator $V$-algebra. By the sake of comfortability, we defined the natural involution of $A$ only under the requirement that $A$ has a right approximate identity (cf. Definition \ref{definition2}). But, regarding assertions (i) and (iii) in Fact \ref{fact00}, we realize that such a \textit{natural involution} can be defined verbatim under the weaker and more natural requirement that \linebreak $l$-$\Ann (A)=0$. This will be assumed without notice in what follows.

\textit{Left alternative algebras} over $\K$ are defined as those algebras $A$ over $\K$ such that the equality $a^2b=a(ab)$ holds for all $a,b\in A$. We note that left alternative $*$-algebras are in fact alternative algebras.  Nevertheless, a priori this cannot be applied to left alternative left operator $V$-algebras with zero left annihilator, as we did not require that the natural involutions of such algebras be algebra involutions.

To do understandable the proof of the next theorem, let us recall some terminology.
Let $A$ be an algebra over $\K$. Note at first that Jordan subalgebras of $A$ are precisely those subspaces $C$ of $A$ such that $c^2$ lies in $C$ whenever $c$ belongs to $C$.  Now,
let $A$ and $B$ be algebras over $\K$. By a \textit{Jordan homomorphism} from $A$ to $B$ we mean
an algebra homomorphism from $A^{\rm sym}$ to $B^{\rm sym}$. We note that Jordan homomorphisms
from $A$ to $B$ are precisely those linear mappings $F : A\to B$ such that $F(a^2) = F(a)^2$
for every $a\in A$.

Now we can prove the following refinement of Corollary ~\ref{alternativeunit-freeVP}.

\begin{theorem} \label{finalfinal}
For a complete normed complex algebra $A$, the following conditions are equivalent:
\begin{enumerate}
\item[\rm (i)] $A$ is a nonzero alternative $C^*$-algebra.
\item[\rm (ii)] $A$ is an alternative left operator $V$-algebra having an approximate identity bounded by $1$ $($endowing $A$ with its natural involution$)$.
\item[\rm (iii)] $A$ is a  left alternative left operator $V$-algebra having a right approximate identity bounded by $1$ $($endowing $A$ with its natural involution$)$.
\item[\rm (iv)] $A$ is a  left alternative left operator $V$-algebra such that the equality $\|L_a^A\|=\|a\|$ holds for every $a\in A$ $($endowing $A$ with its natural involution, which can be considered because in this case the equality $l$-$\Ann (A)=0$ is clear$)$.
\end{enumerate}
\end{theorem}

\begin{proof}
(i)$\Rightarrow $(ii) By the easy part of Corollary \ref{alternativeunit-freeVP}.

(ii)$\Rightarrow $(iii) This is clear.

(iii)$\Rightarrow $(iv) By assertion (ii) in Fact \ref{fact00}.

(iv)$\Rightarrow $(i) Suppose that $A$ is a left alternative complete left operator $V$-algebra such that $\|L_a^A\|=\|a\|$ for every $a\in A$.
Let $\Phi :A\to BL(A)$ be defined by $\Phi (a):=L_a^A$ for every $a\in A$. Then clearly $\Phi $ is a linear isometry. On the other hand, by the definition itself of left alternative algebras, for every $a\in A$ we have $L_{a^2}^A=(L_a^A)^{2}$, and hence $\Phi $ is a Jordan homomorphism.  It follows that the set
\[
L:=\{\lambda I_A+\Phi (a):(\lambda ,a)\in \C \times A\}
\]
 is a closed Jordan subalgebra of $BL(A)$. Moreover, since $A$ is a left operator $V$-algebra, it follows from the local behaviour of numerical ranges \cite[Corollary 2.1.2]{CR} that $L$
is a $V$-algebra, and that $\Phi $, regarded now as a mapping from $A$ to $L$, is a $*$-mapping when we endow $A$ and $L$ with their respective natural involutions. In this way,  $A^{\rm sym}$ and $\Phi (A)$ are identified as Banach spaces, as Jordan algebras, and as involutive complex vector spaces.
But, by the non-associative Vidav--Palmer Theorem \ref{non-associativeVP}, $L$ is a $JB^*$-algebra. Therefore, since $\Phi (A)$ is a $JB^*$-algebra (as it is a closed $*$-invariant subalgebra of $L$), we conclude that  $A^{\rm sym}$ is a $JB^*$-algebra.
It follows that $A$ is a $JB^*$-admissible algebra in the sense of \cite[Definition 3.5.29]{CR}, and hence, by \cite{R91} (cf. \cite[Proposition 3.5.31]{CR}), $A$ is a non-commutative $JB^*$-algebra. As a byproduct, $A$ is a left alternative $*$-algebra, and hence $A$ is alternative. But, as we already know, non-commutative $JB^*$-algebras which are alternative are precisely the alternative $C^*$-algebras.
\end{proof}

The equivalence (i)$\Leftrightarrow $(ii) in Theorem \ref{finalfinal} was previously proved in Corollary \ref{alternativeunit-freeVP}. Anyway, the proof immediately above provides us with a new proof of that equivalence, whose method is completely different from that applied in the proof of Corollary \ref{alternativeunit-freeVP}.

The next proposition follows straightforwardly from Theorem \ref{finalfinal}.

\begin{proposition} \label{finalfinalbis}
For a complete normed alternative complex algebra $A$, the following conditions are equivalent:
\begin{enumerate}
\item[\rm (i)] $A$ is a nonzero alternative $C^*$-algebra.
\item[\rm (ii)] $A$ is a left operator $V$-algebra having a right approximate identity bounded by $1$ $($endowing $A$ with its natural involution$)$.
\item[\rm (iii)] $A$ is a left operator $V$-algebra such that the equality $\|L_a^A\|=\|a\|$ holds for every $a\in A$ $($endowing $A$ with its natural involution$)$.
\end{enumerate}
\end{proposition}

We note that the specialization of the above proposition in the associative context  provides us with a refinement of Theorem \ref{unit-freeVP}. In particular we are provided with the following.

\begin{corollary} \label{unit-freeVPbis}
Nonzero $C^*$-algebras are precisely the complete associative left operator $V$-algebras having a right approximate identity bounded by $1$, endowed with their natural involutions.
\end{corollary}

To reach the above corollary we have applied the non-associative Vidav--Palmer Theorem \ref{non-associativeVP} and other nontrivial results of the theory of non-commutative $JB^*$-algebras (see the proof of Theorem \ref{finalfinal}). Nevertheless, as we show immediately below,  Corollary \ref{unit-freeVPbis} has an autonomous proof involving only associative arguments.

\begin{proof}[Autonomous proof of Corollary \ref{unit-freeVPbis}]
Let $A$ be a complete associative left operator $V$-algebra having a right approximate identity bounded by $1$. Let $\Phi :A\to BL(A)$ be defined by $\Phi (a):=L_a^A$ for every $a\in A$. Then, by Fact \ref{fact00}(ii), $\Phi $ is a linear isometry. On the other hand, since $A$ is associative, $\Phi $ is an algebra homomorphism.  It follows that the set
\[
L:=\{\lambda I_A+\Phi (a):(\lambda ,a)\in \C \times A\}
\]
 is a closed subalgebra of $BL(A)$. Moreover, since $A$ is a left operator $V$-algebra, $L$
is a $V$-algebra, and $\Phi $, regarded now as a mapping from $A$ to $L$, is a $*$-mapping when we endow $A$ and $L$ with their respective natural involutions. In this way, $A$ and $\Phi (A)$ are identified as Banach spaces, as associative algebras, and as involutive complex vector spaces.
But, by the associative Vidav--Palmer Theorem \ref{VP}, $L$ is a $C^*$-algebra. Therefore, since $\Phi (A)$ is a $C^*$-algebra (as it is a closed $*$-invariant subalgebra of $L$), we conclude that $A$ is a $C^*$-algebra.
\end{proof}

Another straightforward consequence of Proposition \ref{finalfinalbis} deserving some emphasis is the following.

\begin{corollary} \label{unit-freeVPtris}
Unital $C^*$-algebras $($respectively, unital alternative $C^*$-al\-gebras$)$ are precisely the complete associative $($respectively, alternative$)$ left operator $V$-algebras having a norm-one right unit, endowed with their natural involutions.
\end{corollary}

We do not know if Proposition \ref{finalfinalbis} remains true with ``possibly non-asso\-ciative" instead of ``alternative" at the beginning, and ``non-commutative $JB^*$-algebra" instead of ``alternative $C^*$-algebra" in condition (i), nor even if we replace ``alternative" with ``non-commutative Jordan" at the beginning, if we replace ``alternative $C^*$-algebra" with ``non-commutative $JB^*$-algebra" in condition (i), and if we erase condition (iii).

\section{Beginning a general theory of left operator $V$-algebras}

Left operator $V$-algebras have played a central role in this paper, so one is tempted to tray to develop a general theory for them. As a first sample in this line, we are provided with the following.

\begin{proposition} \label{lunes}
Let $A$ be a complete left operator $V$-algebra with zero left annihilator. Then the involution of $A$ is continuous.
\end{proposition}
\begin{proof}
Since $LOH(A)$ is a closed real subspace of $A$, it follows from the completeness of $A$ that the direct sum $A=LOH(A)\oplus iLOH(A)$ is topological.
\end{proof}

Now we are going to deal with stability properties of the class of all left operator $V$-algebras.

The following fact follows straightforwardly from \cite[Lemma 2.2.24]{CR}.

\begin{fact} \label{factsubalgebra}
Let $A$ be a normed algebra over $\K$, and let $B$ be a nonzero subalgebra of $A$. Then for $b\in B$ we have $LOV(B,b)\subseteq LOV(A,b)$.
Therefore, if $\K =\C$, then  $LOH(A)\cap B\subseteq LOH(B)$.
\end{fact}

With the above fact in mind, the next proposition is immediate.

\begin{proposition} \label{starsubalgebra}
Let $A$ be a left operator $V$-algebra with zero left annihilator, and let $B$ be a nonzero subalgebra of $A$ invariant under the natural involution of $A$. Then $B$ is a left operator $V$-algebra. Moreover, if in addition $B$ has zero left annihilator in itself, then the natural involution of $B$ coincides with the restriction to $B$ of the natural involution of $A$.
\end{proposition}

\begin{corollary} \label{vaya}
Let $A$ and $B$ be nonzero normed algebras over $\K$, and let \linebreak $\phi :A\to B$ be a contractive algebra homomorphism. Suppose that $A$ has an approximate identity bounded by $1$. Then for $a\in A$ we have
\[
LOV(B,\phi (a))\subseteq LOV(A,a).
\]
Therefore $\phi (LOH(A))\subseteq LOH(B)$.
\end{corollary}
\begin{proof}
Let $a_\lambda $ be the approximate identity bounded by $1$ in $A$ whose existence has been assumed. Then
$\phi (A)$ is a subalgebra of $B$, and $\phi (a_\lambda )$ is an approximate identity bounded by $1$ in $\phi (A)$. Therefore, by Lemma \ref{4.2bis}, we have $LOV(B,\phi (a))\subseteq LOV(\phi (A),\phi (a))$.
But, by Fact \ref{factsubalgebra}, the inclusion $LOV(\phi (A),\phi (a))\subseteq LOV(A,a)$ holds.
\end{proof}

\begin{remark} \label{easypart}
According to the easy part of Theorem \ref{unit-freenon-associativeVP}, non-commutative $JB^*$-algebras are  left operator $V$-algebras having approximate identities bounded by $1$, and their involutions as non-commutative $JB^*$-algebras coincide their natural involutions as left operator $V$-algebras with zero left annihilator.
\end{remark}

Now, considering Corollary \ref{vaya} and Remark \ref{easypart}, we obtain the following.

\begin{corollary} \label{Peralta}
{\rm \cite{EPR} (cf. \cite[Proposition 5.9.3]{CRbis})}
Contractive algebra homomorphisms between non-commu\-tative $JB^*$-algebras are $*$-mappings.
\end{corollary}

\begin{proposition} \label{estoycansad}
Let $A$ be a non-nilpotent complete normed complex algebra having an approximate identity bounded by $1$, and let $D$ be a continuous derivation of $A$ such that $0$ is not an interior point of $V(BL(A),I_A,D)$. Then there is a module-one  complex number $\alpha $ such that $\alpha D(LOH(A))\subseteq LOH(A)$.
\end{proposition}
\begin{proof}
Set $K:=V(BL(A),I_A,D)$. Then by \cite[Proposition 10.6]{BD} (see also \cite[Lemma 2.3.21]{CR}) and \cite{AR} (cf. \cite[Corollary 3.4.44(i)]{CR}), $0\in K$. Therefore, since $0$ is not an interior point of $K$, there is a module-one  complex number $\alpha $ such that $\Re (z)\leq 0$ for every  $z\in \alpha K$. It follows from \cite[Corollary 10.13 and Proposition 18.7]{BD} (see also \cite[Corollary 2.1.9(i) and Lemma 2.2.21]{CR}) that $\exp (\alpha D)$ is a contractive algebra automorphism of $A$. Now, by Corollary \ref{vaya}, we have $\exp (\alpha D)(LOH(A))\subseteq LOH(A)$. Finally, arguing as at the end of the proof of Proposition \ref{propisition1}, we conclude that $\alpha D(LOH(A))\subseteq LOH(A)$.
\end{proof}

 Let $X$ be a complex vector space  endowed with a conjugate-linear involution~$*$, and let $L(X)$ denote the algebra of all linear operators on $X$. \textit{For $F\in L(X)$ define $F^*\in L(X)$  by $F^*(x):=(F(x^*))^*$.}  Then it is straightforward that the mapping $F\to F^*$ from $L(X)$ to $L(X)$ is a conjugate-linear vector space involution satisfying $(F\circ G)^*=F^*\circ G^*$ for all $F,G\in L(X)$.

We recall that derivations of non-commutative $JB^*$-algebras are automatically continuous \cite{Y} (cf. \cite[Lemma 3.4.26]{CR}).

\begin{theorem} \label{new}
Let $A$ be a nonzero non-commutative $JB^*$-algebra, and let $D$ be a derivation of $A$ such that $0$ is not an interior point of $V(BL(A),I_A,D)$. Then there is a module-one  complex number $\alpha $ such that $D^*=\alpha ^2D$, and actually $V(BL(A),I_A,D)$ is contained in a straight line passing by $0$. More precisely, we have $V(BL(A),I_A,D)= i\overline{\alpha} \, [-\|D\|,\|D\|]$.
\end{theorem}
\begin{proof}
Clearly, $A$ is not nilpotent. Therefore, by Remark \ref{easypart} and Proposition \ref{estoycansad}, there is a module-one complex number $\alpha $ such that $\alpha D(LOH(A))\subseteq LOH(A)$. This implies that $(\alpha D)^*=\alpha D$, and hence $D^*=\alpha ^2 D$. Moreover, since $(-i\alpha D)^*=i\alpha D$, we can argue as in the proof of \cite[Lemma 3.4.77]{CR}, with $i\alpha D$ instead of $D$, to obtain that $V(BL(A),I_A,i\alpha D)\subseteq \R$ and that
$\{-\|D\|,\|D\|\}\subseteq {\rm sp}(BL(A),i\alpha D)$. But, by \cite[Lemma 2.3.21]{CR}, we have $ {\rm sp}(BL(A),i\alpha D)\subseteq V(BL(A),I_A,i\alpha D)$. Therefore, since the inclusion $V(BL(A),I_A,i\alpha D)\subseteq \|D\|\B _\C $ is clear, it follows that
$V(BL(A),I_A,i\alpha D)=[-\|D\|,\|D\|]$, or equivalently $V(BL(A),I_A,D)= i\overline{\alpha} \, [-\|D\|,\|D\|]$.
\end{proof}

\begin{remark}
In the proof of Theorem \ref{new} we have applied that non-commutative $JB^*$-algebras are not nilpotent. Actually, as we are going to show here, \textit{every left operator $V$-algebra with zero left annihilator is not nilpotent.} Let $A$ be such an algebra. Then, since $A$ is nonzero, and is the complex linear hull of $LOH(A)$, and has zero left annihilator, there must exist $h\in LOH(A)$ such that $L_h^A\neq 0$. Therefore, since  $L_h^A\in H(BL(A),I_A)$, it follows from \cite[Theorem 10.17]{BD} (see also \cite[Proposition 2.3.22]{CR}) that $(L_h^A)^n\neq 0$ for every $n\in \N$. Thus, for each $n\in \N$ there exists $a_n\in A$ such that $(L_h^A)^n(a_n)\neq 0$. This implies that $A$ is not nilpotent.

\end{remark}

The following fact follows straightforwardly from \cite[Lemma 5.7.1]{CRbis}.

\begin{fact} \label{factquotient}
Let $A$ be a normed algebra over $\K$, let $M$ be a closed proper ideal of $A$, and let $q:A\to A/M$ denote the natural quotient mapping. Then for $a\in A$ we have
$LOV(A/M,q(a))\subseteq LOV(A,a)$.
Therefore, if $\K =\C$, then the inclusion $q(LOH(A))\subseteq LOH(A/M)$ holds.
\end{fact}

If $X$ is a vector space over $\K$, if $*$ is a conjugate-linear
involution on $X$, and if $M$ is a $*$-invariant subspace of $X$,
then $x+M\to x^*+M$ becomes a well-defined conjugate-linear
involution on $X/M$, which will be called {\it the quotient involution}.

\begin{proposition} \lbl{Videal}
Let $A$ be a left operator $V$-algebra, and let $M$ be a closed proper ideal of
$A$. Then $A/M$ is a left operator $V$-algebra. Moreover, if in addition $A$ has a right approximate identity, then:
\begin{enumerate}
\item[\rm (i)] $M$ is invariant under the natural involution of $A$.
\item[\rm (ii)] $A/M$ has a right approximate identity, and the natural involution of $A/M$ is the
quotient of the natural involution of $A$.
\end{enumerate}
\end{proposition}
\begin{proof}
Let $q:A\to A/M$ denote the natural quotient mapping. Let $a$ be in $A$, and write $a=h+ik$ with $h,k\in LOH(A)$. Then $q(a)=q(h)+iq(k)$ with  $q(h),q(k)\in LOH(A/M)$ thanks to fact \ref{factquotient}. Therefore, since $a$ is arbitrary in $A$, and $q$ is surjective, we see that $A/M$ is a left operator $V$-algebra.

Suppose that $A$ has a right approximate identity (say $a_\lambda $). Then clearly $q(a_\lambda )$ is a right approximate identity in $A/M$. Now let $*$ and $\natural $ denote the natural involution of $A$ and of $A/M$, respectively. Let $a$ be in $A$, and write $a=h+ik$ with $h,k\in LOH(A)$.
Then we have
\[
q(a^*)=q(h)-iq(k)=(q(h)+iq(k))^\natural =q(a)^\natural ,
\]
where for the second equality we have applied Fact \ref{factquotient} again. It follows that $a^*$ lies in $M$ whenever $a$ belongs to $M$ (so $M$ is $*$-invariant because $a$ is arbitrary in $A$) and that $\natural $ is the quotient involution of $*$.
\end{proof}

It is clear that, for a normed algebra $A$ over $\K$, the property of having a (two-sided) approximate identity bounded by $1$ passes from $A$ to the quotient of $A$ by any closed ideal of $A$. Considering this, it is enough to combine Theorem \ref{unit-freenon-associativeVP} and Proposition \ref{Videal} to obtain the following.

\begin{corollary}
{\rm \cite[Proposition 3.4.13]{CR}}
Let $A$ be a non-commutative $JB^*$-algebra, and let $M$ be a closed ideal of $A$. Then $M$ is $*$-invariant, and $A/M$, endowed with the quotient norm and the quotient involution, becomes a non-commutative $JB^*$-algebra.
\end{corollary}

\begin{lemma} \label{linfinite}
Let $I$ be a nonempty set, let $\{X_i \}_{i \in I}$ be a family
of nonzero normed spaces over $\K$, and, for  $i\in I$, let $F_i$ be in $BL(X_i)$ in such a way that $\sup \{\|F_i\|:i\in I\}<\infty $.
Denote by $X$ the $\ell_\infty$-sum
of the family $\{X_i \}_{i \in I}$, and let $F\in BL(X)$ be defined by $F(\{x_i\}):=\{F_i(x_i)\}$ for every $\{x_i\}\in X$. Then the equality
\[
V(BL(X),I_X,F) = \overline{\rm co} \, [\cup _{i\in
I} V(BL(X_i), I_{X_i}, F_i)]
\]
 holds, where $\overline{\rm co}$ means closed convex hull.
\end{lemma}
\begin{proof}
Let $\mathcal A$ denote the $\ell_\infty$-sum
of the family $\{BL(X_i)\}_{i \in I}$. Then, since $\{F_i\}$ is an arbitrary element of $\mathcal A$, we can let it run over $\mathcal A$, and consider the mapping $\{F_i\}\to F$ from $\mathcal A$  to $BL(X)$. It is easily realized that such a mapping is a linear isometry taking $\{I_{X_i}\}$ to $I_X$. Therefore, by \cite[Corollary 2.1.2(ii)]{CR}, we have $V(BL(X),I_X,F) = V(\mathcal A,\{I_{X_i}\},\{F_i\})$. Now the result follows from \cite{R1} (cf. \cite[Corollary 2.9.50]{CR}).
\end{proof}

The above lemma straightforwardly yields the following.

\begin{proposition} \label{ya}
Let $I$ be a nonempty set, and let $\{A_i \}_{i \in I}$ be a family
of nonzero normed algebras over $\K$.
Denote by $A$ the normed algebra $\ell_\infty$-sum
of the family $\{A_i \}_{i \in I}$. Then, for $\{a_i\}\in A$, we have
$$LOV(A,\{a_i\})
= \overline{\rm co} \, [\cup _{i\in
I} LOV(A_i,a_i)],$$
and hence $LOH(A)$ is equal to the $\ell_\infty$-sum
of the family $\{LOH(A_i) \}_{i \in I}$. Therefore, if $A_i$ is a left operator $V$-algebra for every $i\in I$, then $A$ is a left operator $V$-algebra. Moreover, if in addition $A_i$ has zero left annihilator for every $i\in I$, then $A$ has zero left annihilator, and the equality $\{a_i\}^*=\{a_i^*\}$ holds for every $\{a_i\}\in A$, where $*$ stands indistinctly for the natural involution of $A$ and that of each $A_i$.
\end{proposition}

\begin{lemma} \label{bycontinuity}
Let $X$ be a normed space over $\K$, and let $\widehat{X}$ denote the completion of $X$. For $F\in BL(X)$, let $\widehat{F}\in BL(\widehat{X})$ denote the extension of $F$ by continuity. Then for every $F\in BL(X)$ we have $$V(BL(\widehat{X}),I_{BL(\widehat{X)}},\widehat{F})=V(BL(X),I_X,F).$$ As a consequence, if $A$ is a normed complex algebra, then, denoting by $\widehat{A}$ the algebra completion of $A$, we have $LOH(A)=A\cap LOH(\widehat{A})$.
\end{lemma}
\begin{proof}
The mapping $F\to \widehat{F}$ is a linear isometry taking $I_X$ to $I_{\widehat{X}}$, and \cite[Corollary 2.1.2(ii)]{CR} applies.
\end{proof}

\begin{lemma} \label{completion00}
Let $A$ be a left operator $V$-algebra such that there exists $M\geq 1$ satisfying $\|a\|\leq M\|L_a\|$ for every $a\in A$. Then the natural involution of $A$ is continuous, and the completion of $A$ is a left operator $V$-algebra whose natural involution extends that of $A$.
\end{lemma}
\begin{proof}
Let $h$ and $k$ be in $LOH(A)$. Then, by  \cite[Corollary 2.3.5]{CR}, we have
$\|h\|\leq M\|L_h\|\leq M\|L_h+iL_k\|\leq M\|h+ik\|$,
Therefore
\begin{equation} \label{direcsum}
\mbox{the direct sum $A=LOH(A)\oplus LOH(A)$ is topological,}
\end{equation}
and hence the involution natural of $A$ is continuous. Let $\widehat{A}$ denote the completion of $A$, and let $a$ be in $\widehat{A}$. Then there exists a sequence $a_n$ in $A$ such that $\lim _na_n=a$. For each $n\in \N$, write $a_n=h_n+ik_n$ with $h_n,k_n\in LOH(A)$. Then, by \eqref{direcsum}, $h_n$ and $k_n$ are Cauchy sequences in $A$, and hence they converge in $\widehat{A}$ to some $h$ and $k$ respectively. But, by Lemma \ref{bycontinuity}, $h_n,k_n\in LOH(\widehat{A})$.
Therefore, since $LOH(\widehat{A})$ is closed in $\widehat{A}$, we see that $a=h+ik$ with $h,k\in LOH(\widehat{A})$. Since $a$ is arbitrary in $\widehat{A}$, we conclude that $\widehat{A}$ is a left operator $V$-algebra. The remaining part of the lemma follows from Lemma \ref{bycontinuity}.
\end{proof}

\begin{lemma} \label{approximate}
Let $A$ be a normed algebra over $\K$, and let $B$ be a dense subalgebra of $A$. Then every bounded right approximate identity in $B$ is a right approximate identity in $A$ .
\end{lemma}
\begin{proof}
Let $b_\lambda $ be a right approximate identity in $B$ bounded by $M\geq 1$. Let $a$ be in $A$, and let $\varepsilon >0$.
Then there exists $b\in B$ such that $\|a-b\|\leq \frac{\varepsilon }{2(1+M)}$ and, once such a $b$ has been chosen, there exists $\lambda _0$ such that $\|b-bb_\lambda \|\leq \frac{\varepsilon }{2}$ whenever $\lambda \geq \lambda _0$.
It follows from the triangle inequality that
\[
\|a-ab_\lambda \|\leq \|a-b\|+\|b-bb_\lambda \|+\|(b-a)b_\lambda \|
\]
\[
\leq (1+M)\|a-b\|+\|b-bb_\lambda \|\leq \frac{\varepsilon}{2}+\frac{\varepsilon}{2}=\varepsilon
\]
whenever $\lambda \geq \lambda _0$.
\end{proof}

The next proposition becomes a unit-free version of \cite[Corollary 3.3.15]{CR}.

\begin{proposition} \label{completion}
Let $A$ be a left operator $V$-algebra having a bounded right approximate identity. Then the natural involution of $A$ is continuous, and the completion of $A$ is a left operator $V$-algebra having a bounded right approximate identity $($by the way, the same that $A$ had$)$, and whose natural involution extends that of $A$.
\end{proposition}
\begin{proof}
Take $M\geq 1$ such that $\|a_\lambda \|\leq M$ for every $\lambda $. Let $h$ and $k$ be in $LOH(A)$. Then, by Fact \ref{fact00}(ii), we have $\|a\|\leq M\|L_a^A\|$ for every $a\in A$. Therefore, by Lemma \ref{completion00}, the natural involution of $A$ is continuous, and the completion of $A$ is a left operator $V$-algebra whose natural involution extends that of $A$. The remaining part of the proposition follows from Lemma \ref{approximate}.
\end{proof}

\begin{lemma} \label{associative}
Let $A$ be a normed algebra over $\K$ having a right approximate identity $($eventually bounded by $M\geq 1$$)$, and let $B$ be a dense subalgebra of~$A$. Then there exists a right approximate identity in $B$ $($eventually bounded by $M$$)$.
\end{lemma}
\begin{proof}
Let $a_\lambda $ be the right approximate identity in $A$ $($eventually bounded by~$M$$)$ whose existence has been assumed. Let $\mathcal C $ stand for $B$ (eventually, for $M\B _B$).
Then, since $\mathcal C$ is convex, and, $a_\lambda $ lies in the closure of $\mathcal C$ in $A$  for every~$\lambda $, if follows from \cite[Lemma 8.1.138]{CRbis} that there exists a right approximate identity in $A$  consisting of elements of $\mathcal C$.
\end{proof}

\textit{Pre-$C^*$-algebras} (respectively, \textit{alternative pre-$C^*$-algebras}) are defined as $C^*$-algebras (respectively, alternative $C^*$-algebras), but dispensing the requirement of completeness.

Now we can prove the completeness-free version of Proposition \ref{finalfinalbis}.

\begin{theorem} \label{completeness-freeVP}
For a normed alternative complex algebra $A$, the following conditions are equivalent:
\begin{enumerate}
\item[\rm (i)] $A$ is a nonzero alternative pre-$C^*$-algebra.
\item[\rm (ii)] $A$ is a left operator $V$-algebra having a right approximate identity bounded by $1$ $($endowing $A$ with its natural involution$)$.
\item[\rm (iii)] $A$ is a left operator $V$-algebra such that the equality $\|L_a^A\|=\|a\|$ holds for every $a\in A$ $($endowing $A$ with its natural involution$)$.
\end{enumerate}
\end{theorem}
\begin{proof}
(i)$\Rightarrow $(ii) Suppose that (i) holds. Then the completion of $A$ (say $\widehat{A}$) is an alternative $C^*$-algebra. Therefore, by the easy part of Theorem \ref{unit-freeVP}, $\widehat{A}$ is a left operator $V$-algebra, and has an approximate identity bounded by $1$. It follows from Lemma \ref{associative} that there is a right approximate identity in $A$ bounded by $1$. Moreover, by Proposition \ref{starsubalgebra}, $A$ is a left operator $V$-algebra.

(ii)$\Rightarrow $(iii) By Fact \ref{fact00}(ii).

(iii)$\Rightarrow $(i) Suppose that (iii) holds. Then it is easily  verified that the equality $\|L_a^{\widehat{A}}\|=\|a\|$ holds for every $a\in \widehat{A}$. Therefore, by Lemma \ref{completion00}, $\widehat{A}$ is a left operator $V$-algebra.
 It follows from the implication (iii)$\Rightarrow $(i) in Proposition \ref{finalfinalbis} that $\widehat{A}$ is an alternative $C^*$-algebra, and hence $A$ is a pre-$C^*$-algebra.
\end{proof}

Of course, the above theorem remains true with ``associative" instead of ``alternative" at the beginning, and ``pre-$C^*$-algebra" instead of ``alternative pre-$C^*$-algebra" in condition (i). In particular we are provided with the following.

\begin{corollary} \label{unit-freeVP00}
Nonzero pre-$C^*$-algebras are precisely the associative left operator $V$-algebras having a right approximate identity bounded by $1$, endowed with their natural involutions.
\end{corollary}

Now we are going to show that left operator $V$-algebras behave reasonably well concerning the so-called ``contractive projection problem". To this end, we begin by formulating the following fact, whose proof is straightforward.

\begin{fact} \label{factotum}
Let $X$  be a normed space over $\K$, let $\pi :X\to X$ be a contractive linear projection, set $Y:=\pi (X)$, and for $F\in BL(Y)$ let
$\Phi (F)\in BL(X)$
 be defined by $\Phi (F)(x):=F(\pi (x))$. Then $\Phi :BL(Y)\to BL(X)$ is an isometric algebra homomorphism whose range is equal to $\pi \circ BL(X)\circ \pi $.
\end{fact}

 In what follows we will apply only that $\Phi $ is a linear isometry such that $\Phi (I_Y)=\pi $.

\begin{lemma} \label{porvafor2}
Let $A$ be a complex normed algebra, let $\pi
:A\to A$ be a nonzero contractive  linear projection satisfying the ``left weak conditional expectation"
\begin{equation} \label{Hamana}
\mbox{$\pi (\pi (a)\pi(b))=\pi (a\pi (b))$ \ for all \ $a,b\in A$},
\end{equation}
and let $B:=(\pi (A),\odot^\pi )$ denote the normed algebra consisting of the normed space of $\pi (A)$ and the product $x\odot^\pi y:=\pi (xy)$. Then the inclusion
\[
LOV(B,\pi (a))\subseteq LOV(A,a)
\]
holds for every $a\in A$. Therefore $\pi (LOH(A))\subseteq LOH(B)$.
\end{lemma}
\begin{proof}
Let $\Phi :BL(B)\to BL(A)$ be the isometric algebra homomorphism given by Fact \ref{factotum} (when $(A,B)$ replaces $(X,Y)$). Then, as observed at the beginning of the proof of \cite[Lemma 3.2]{R99}, for $x\in B$ we have $\Phi (L_x^B)=\pi \circ L_x^A\circ \pi $. Now let $a$ be in $A$, and note that \eqref{Hamana} can be rewritten as that $\pi \circ L_{\pi (a)}^A\circ \pi =\pi \circ L_a^A\circ \pi $ for every $a\in A$. Since $\Phi (I_B)=\pi $, it follows from \cite[Corollary 2.1.2(ii)]{CR} that
\[
LOV(B,\pi (a))=V(BL(A),\pi ,\Phi (L_{\pi (a)}^B))=V(BL(A),\pi ,\pi \circ L_a^A\circ \pi ).
\]
But, since the mapping $G\to \pi \circ G\circ \pi $ is a linear contraction from $BL(A)$ to $\pi \circ BL(A)\circ \pi $ taking $I_A$ to $\pi $, it follows from \cite[Corollary 2.1.2(i)]{CR} that $V(BL(A),\pi ,\pi \circ L_a^A\circ \pi )\subseteq LOV(A,a)$, which concludes the proof.
\end{proof}

 Now we can prove the main result in this section, namely the following.

\begin{theorem} \label{martes}
Let $A$ be a left operator $V$-algebra, let  $\pi
:A\to A$ be a nonzero contractive  linear projection satisfying
the ``weak conditional expectation"
\begin{equation} \label{pasadomagnana}
\mbox{$\pi (\pi (a)b)=\pi (\pi (a)\pi(b))=\pi (a\pi (b))$ \ for all \ $a,b\in A$},
\end{equation}
and set $B:=(\pi (A),\odot ^\pi )$. Then $B$ is a left operator $V$-algebra. Moreover, if in addition $A$ has a right approximate identity $($eventually bounded by $1$$)$, then we have:
\begin{enumerate}
\item[\rm (i)] $B$ has a right approximate identity $($eventually bounded by $1$$)$, and the natural involution of $B$ $($say $\natural $$)$ is given by $x^\natural =\pi (x^*)$, where $*$ denotes the natural involution of $A$.
\item[\rm (ii)] $\pi =\pi \circ \pi ^* $, hence $\pi ^*=\pi ^*\circ \pi $, and $\ker (\pi ^*)=\ker (\pi )$ is a $*$-invariant subspace of $A$.
\item[\rm (iii)] If $*$ is an isometry, then every convex combination of $\pi $ and $\pi ^*$ is a contractive linear projection on $A$.
 \end{enumerate}
 \end{theorem}
\begin{proof}
Let $x$ be in $B$. Then $x=h+ik$ for suitable $h,k\in LOH(A)$. Therefore $x=\pi (h)+i\pi (k)$, with $\pi (h),\pi (k)\in LOH(B)$ thanks to Lemma \ref{porvafor2}. Since $x$ is arbitrary in $B$, this shows that $B$ is a left operator $V$-algebra.

Suppose that $A$ has a right approximate identity $a_\lambda $ (eventually bounded by $1$). Then, by the first equality in \eqref{pasadomagnana}, for $a\in A$ we have
\[
\mbox{$\lim _\lambda \pi (a)\odot^\pi \pi(a_\lambda )=\lim _\lambda \pi (\pi (a)\pi(a_\lambda ))=\lim _\lambda \pi (\pi (a)a_\lambda )=\pi (\pi (a))=\pi (a).$}
\]
Therefore $\pi (a_\lambda )$ is a right approximate identity (eventually bounded by $1$) for $B$. Let $x$ be in $B$. Write $x=h+ik$  with $h,k\in LOH(A)$. Then we have
\[
\pi (x^*)=\pi (h)-i\pi (k)=(\pi (h)+i\pi (k))^\natural =\pi (x)^\natural =x^\natural ,
\]
where for the second equality we have applied Lemma \ref{porvafor2} again. Thus (i) has been proved.
Now let $a$ be in $A$. Then, by Lemma \ref{porvafor2} once more, we have $\pi (a^*)=(\pi (a))^\natural $. Therefore, replacing $a$ with $a^*$, and considering~(i), we obtain
$\pi (a)=(\pi (a^*))^\natural =\pi ((\pi (a^*))^*)=\pi (\pi ^*(a))$, which proves (ii) because $a$ is arbitrary in $A$. Now suppose additionally that $*$ is an isometry. Then clearly $\pi^*$ is a contractive linear projection on $A$, so every  convex combination of $\pi $ and $\pi ^*$ is a contractive linear operator on $A$. Let $\alpha ,\beta $ be nonnegative real numbers with $\alpha +\beta =1$. Then, considering (ii), we obtain
\[
(\alpha \pi +\beta \pi^* )^2=\alpha (\alpha \pi +\beta \pi \circ \pi ^*)+\beta (\alpha \pi ^*\circ \pi +\beta \pi ^*)=\alpha \pi +\beta \pi^*.
\]
This concludes the proof of (iii) and of the theorem.
\end{proof}

\begin{corollary} \label{projectiontris}
{\rm \cite[Theorem 4.9]{CR16}}
Let $A$ be a non-commutative $JB^*$-algebra, and let $\pi
:A\to A$ be a contractive linear projection satisfying the ``Jordan weak conditional expectation"
\begin{equation} \lbl{magnana}
\mbox{$\pi (\pi (a)\bullet \pi (b))=\pi (a\bullet \pi (b))$  \ for all $a,b\in A$.}
\end{equation}
 Then we have:
\begin{enumerate}
\item[\rm (i)] $\pi (A)$
becomes a non-commu\-tative   $JB^*$-algebra under
the norm of $A$, the product $\odot ^\pi $, and the involution $x^\natural := \pi(x^*)$.
\item[\rm (ii)] $\pi =\pi \circ \pi ^*$, hence $\pi ^*=\pi ^*\circ \pi $, $\ker (\pi ^*)=\ker (\pi )$ is a $*$-invariant subspace of $A$, and every convex combination of $\pi $ and $\pi ^*$ is a contractive linear projection on $A$.
\end{enumerate}
\end{corollary}

\begin{proof}[New proof]
Note that, thanks to the commutativity of the product $\bullet $, our assumption \eqref{magnana} is equivalent to the fact that \eqref{pasadomagnana} holds with $A^{\rm sym}$ instead of $A$. Therefore, since $A^{\rm sym}$ is a $JB^*$-algebra \cite[Fact~3.3.4]{CR}, and the involution of a $JB^*$-algebra is an isometry \cite{Y0} (cf. \cite[Proposition~3.3.13]{CR}), we can combine the commutative particularizations of Theorems \ref{unit-freenon-associativeVP} and \ref{martes} to obtain that $(\pi (A),\odot ^\pi )^{\rm sym}$ is a $JB^*$-algebra under
the norm of $A$ and the involution $x^\natural := \pi(x^*)$, and that moreover assertion (ii) holds for $A$.
Now $(\pi (A),\odot ^\pi )$ is a complex normed algebra such that $((\pi (A),\odot )^{\rm sym},\natural )$ is a $JB^*$-algebra, and therefore, by \cite{R91} (cf. \cite[Definition 3.5.29 and Proposition 3.5.31]{CR}), $((\pi (A),\odot ^\pi ),\natural )$ is a non-commutative $JB^*$-algebra.
\end{proof}

Corollary \ref{projectiontris} is only a part of \cite[Theorem 4.9]{CR16}. Thus, among other additional results, \cite[Theorem 4.9]{CR16} assures that Corollary \ref{projectiontris} remains true with ``commutative $C^*$-algebra" instead of ``non-commutative $JB^*$-algebra" in both the assumption and the conclusion. Nevertheless, according to
\cite[Corollary 4.15]{CR16}, Corollary \ref{projectiontris} does not remain true with ``$C^*$-algebra" instead of ``non-commutative $JB^*$-algebra" in both the assumption and the conclusion. Moreover, according to \cite[Proposition 4.14]{CR16}, even if in Corollary \ref{projectiontris} the non-commutative $JB^*$-algebra $A$ is in fact  a commutative $C^*$-algebra, we cannot expect that the range of $\pi $ be $*$-invariant.

\section*{Acknowledgements}
The author wants to thank Miguel Cabrera for his careful reading of the various previous drafts of the current paper, and for his useful suggestions to improve them.

Special thanks are due to the referees for careful reading of the paper and useful suggestions.

\end{document}